\documentclass[a4paper,11pt]{article}
\usepackage{latexsym}
\usepackage{graphicx}

\usepackage{amsmath, amsthm, amsfonts, amssymb, bbm}
\usepackage[mathscr]{eucal}

\theoremstyle{plain}
\newtheorem{proposition}{Proposition}[section]
\newtheorem{theorem}[proposition]{Theorem}
\newtheorem{lemma}[proposition]{Lemma}
\newtheorem{corollary}[proposition]{Corollary}

\theoremstyle{definition}
\newtheorem{definition}[proposition]{Definition}
\newtheorem{example}[proposition]{Example}

\theoremstyle{remark}
\newtheorem{remark}[proposition]{Remark}

\newcommand{\rhs}{r.h.s.\ }

\newcommand{\wrt}{w.r.t.\ }

\newcommand{\ud}{\mathrm{d}}
\newcommand{\del}{\partial}

\newcommand{\skal}[2]{\langle #1 , #2 \rangle}

\DeclareMathOperator{\sign}{sign}

\DeclareMathOperator{\supp}{supp}

\newcommand{\V}[1]{\mathbf{#1}}

\newcommand{\R}{\mathbb{R}}
\newcommand{\N}{\mathbb{N}}

\newcommand{\betrag}[1]{\lvert #1 \rvert}
\newcommand{\norm}[1]{\lVert #1 \rVert}

\DeclareMathOperator{\WF}{WF}
\DeclareMathOperator{\SP}{SP}
\DeclareMathOperator{\M}{M}
\newcommand{\D}{\mathcal{D}}
\newcommand{\Sym}[3]{S^{#3}(#1, #2)}

\begin{document}

\title{The wave front set of oscillatory integrals with inhomogeneous phase function}
\author{
Jochen Zahn \\
Courant Research Centre ``Higher Order Structures'' \\
University of G\"ottingen \\
Bunsenstra{\ss}e 3--5, D-37073 G\"ottingen, Germany
}

\maketitle

\begin{abstract}
A generalized notion of oscillatory integrals that allows for inhomogeneous phase functions of arbitrary positive order is introduced. The wave front set of the resulting distributions is characterized in a way that generalizes the well-known result for phase functions that are homogeneous of order one.
\end{abstract}

\section{Introduction}
Oscillatory integrals play an important role in the theory of pseudodifferential operators.
They are also a useful tool in Mathematical Physics, in particular in quantum field theory, where they are used to give meaning to formal Fourier integrals in the sense of distributions. For phase functions which are homogeneous of order one, this also leads to a characterization of the wave front set of the resulting distribution, as it is known to be contained in the manifold of stationary phase of the phase function. In these applications, the restriction to phase functions that are homogeneous of order one is often obstructive. In many cases, this restriction can be overcome by shifting a part of the would-be phase function to the symbol, cf. Example~\ref{ex:Delta+} below. However, such a shift is not always possible, for instance if the would-be phase function contains terms of order greater than one. Such phase functions are present in the twisted convolutions that occur in quantum field theory on Moyal space, cf. Examples~\ref{ex:NCQFT} and \ref{ex:NCQFTb} below. Up to now, a rigorous definition of these twisted convolution integrals could be given only in special cases and in such a way that the information on the wave front set is lost. Thus, it is highly desirable to generalize the notion of oscillatory integrals to encompass also phase functions that are not homogeneous of order one. Such generalizations were proposed by several authors. However, to the best of our knowledge, the wave front sets of the resulting distributions were not considered, except for one very special case. We comment on these settings below, cf. Remark~\ref{rem:AsadaFujiwara}.


It is shown here that the restriction to phase functions that are homogeneous of order one can indeed be weakened, without losing information on the wave front set. The generalization introduced here not only allows for inhomogeneous phase functions, but also for phase functions that are symbols of any positive order. However, one has to impose a condition that generalizes the usual nondegeneracy requirement. It is also shown that the wave front sets of the distributions thus obtained are contained in a set that generalizes the notion of the manifold of stationary phase. We conclude with a discussion of some applications.



Throughout, we use the following notation: For an open set $\Omega$, $K \Subset \Omega$ means that $K$ is a compact subset of $\Omega$. $\dot \R^n$ stands for $\R^n \setminus \{ 0 \}$. For a subset $M \subset \Omega \times \dot \R^m$, $\Pi_1 M$ stands for the projection on the first component. $C^k(\Omega)$ denotes the $k$ times continuously differentiable functions supported on $\Omega$ and $C_0^k(\Omega)$ the set of elements of $C^k(\Omega)$ with compact support in $\Omega$. The dual space of $C^k_0(\Omega)$ is denoted by $\D'^{(k)}(\Omega)$ and $\D'(\Omega) = \D'^{(\infty)}(\Omega)$. The pairing of a distribution $F$ and a test function $f$ is denoted by $\skal{F}{f}$. The dot $\cdot$ stands for the scalar product on $\R^n$. $\sphericalangle(x,y)$ denotes the angle between two vectors $x, y \in \R^n$.

\section{Generalized oscillatory integrals}

As usual, cf. \cite{GrigisSjoestrand}, we define a symbol as follows:
\begin{definition}
Let $\Omega \subset \R^n$ be an open set. A function $a: \Omega \times \R^s \to \mathbb{C}$ is called a \emph{symbol of order $m$} if for each $K \Subset \Omega$ and multiindices $\alpha, \beta$, we have
\[
 \norm{a}_{\alpha \beta K} = \sup_{x \in K, \theta} \betrag{ (D_x^\alpha D_\theta^\beta a)(x, \theta) } \left( 1+\betrag{\theta} \right)^{\betrag{\beta}-m} < \infty.
\]
The set of all such functions, equipped with these seminorms will be denoted by $\Sym{\Omega}{s}{m}$. Furthermore, we denote $\Sym{\Omega}{s}{-\infty} = \cap_m \Sym{\Omega}{s}{m}$ and $\Sym{\Omega}{s}{\infty} = \cup_m \Sym{\Omega}{s}{m}$.
\end{definition}

\begin{remark}
For simplicity, we restrict ourselves to these symbols. The generalization to the symbols $S^m_{\rho \delta}$ is straightforward. One then has to restrict to $1-\mu < \rho \leq 1$, $0 \leq \delta < \mu$, where $\mu$ is the order of the phase function introduced below. Also the generalization to asymptotic symbols as discussed in \cite{RS2} is straightforward.
\end{remark}


The following proposition is a straightforward consequence of the definition of $\Sym{\Omega}{s}{m}$:
\begin{proposition}
\label{prop:cont}
The maps $D^\alpha_x: \Sym{\Omega}{s}{m} \to \Sym{\Omega}{s}{m}$, $D^\beta_\theta: \Sym{\Omega}{s}{m} \to \Sym{\Omega}{s}{m-\betrag{\beta}}$ and the multiplication $\Sym{\Omega}{s}{m} \times \Sym{\Omega}{s}{m'} \to \Sym{\Omega}{s}{m+m'}$ are continuous.
\end{proposition}

The following proposition is proven in \cite[Prop.~1.7]{GrigisSjoestrand}:
\begin{proposition}
\label{prop:dense}
If $m'>m$, then $\Sym{\Omega}{s}{-\infty}$ is dense in $\Sym{\Omega}{s}{m}$ for the topology of $\Sym{\Omega}{s}{m'}$.
\end{proposition}


Now we introduce our new definition of a phase function.
\begin{definition}
\label{def:Phase}
A \emph{phase function} of order $\mu$ on $\Omega \times \R^s$ is a function $\phi: \Omega \times \R^s \to \R$ such that
\begin{enumerate}
 \item $\phi$ is a symbol of order $\mu>0$.
 \item 
\label{enum:Phase}
For each $K \Subset \Omega$ there are positive $C, D$ such that
\[
 \eta(x, \theta) := \betrag{\nabla_x \phi}^2 + \betrag{\theta}^2 \betrag{\nabla_\theta \phi}^2 \geq C \betrag{\theta}^{2\mu} \ \forall \betrag{\theta} \geq D, x \in K.
\]
\end{enumerate}
\end{definition}

\begin{remark}
 \label{rem:Phase}
Condition~\ref{enum:Phase} generalizes the usual nondegeneracy requirement and ensures that $\phi$ oscillates rapidly enough for large $\theta$. In particular it means that $\phi$ is not a symbol of order less than $\mu$. It also means that one can choose $\chi \in C^\infty(\Omega \times \R^s)$ such that $\eta^{-1}(1-\chi)$ is well-defined and a symbol of order $-2\mu$. Here $\chi$ can be chosen such that $\chi(x, \cdot)$ is compact for each $x$.
\end{remark}

\begin{remark}
\label{rem:AsadaFujiwara}
Our definition of a phase function is a generalization of a definition introduced by H\"ormander \cite[Def.~2.3]{Elliptic} in the context of pseudodifferential operators. He considered phase functions of order 1 (in the nomenclature introduced above) and characterized the singular support of the resulting distribution, but not its wave front set. Our characterization of the singular support (cf. Corollary~\ref{cor:M}) coincides with the one given by H\"ormander \cite[Thm.~2.6]{Elliptic}.

Inhomogeneous phase functions were also considered by Asada and Fujiwara \cite{AsadaFujiwara} in the context of pseudodifferential operators on $L^2(\R^n)$. In their setting, $\Omega = \R^{2n}$, $s=n$ and there must be a positive constant $C$ such that
\begin{equation}
\label{eq:AsadaFujiwara}
 \betrag{ \det \begin{pmatrix} \del_x \del_y \phi & \del_x \del_\theta \phi \\ \del_\theta \del_y \phi & \del_\theta \del_\theta \phi \end{pmatrix} } \geq C.
\end{equation}
Furthermore, all the entries of this matrix (and their derivatives) are required to be bounded. Thus, the phase function is asymptotically at least of order 1 and at most of order 2. The admissible amplitudes are at most of order 0. The wave front set of such operators on $L^2(\R^n)$ is not considered by Asada and Fujiwara. The same applies to the works of Boulkhemair \cite{Boulkhemair} and Ruzhansky and Sugimoto \cite{RS06}, who work in a similar context.

Coriasco \cite{Coriasco} considered a special case of H\"ormander's framework, where again $\Omega = \R^{2n}$, $s=n$ and $\phi(x,\theta,y) = \varphi(x,\theta) - y \cdot \theta$ with $\varphi \in \mathbf{SG}^{(1,1)}$, a subset of the symbols of order 1. Furthermore, he imposed growth conditions on $\varphi$ that are more restrictive than Condition~\ref{enum:Phase}. The resulting operators on $\mathcal{S}(\R^n)$ can then be extended to operators on $\mathcal{S}'(\R^n)$. If a further condition analogous to \eqref{eq:AsadaFujiwara} is imposed, then also the wave front set, which is there defined via $\mathbf{SG}$-microregularity, can be characterized (at least implicitly, by the change of the wave front set under the action of the operator).
\end{remark}

\begin{proposition}
\label{prop:Diff}
If $\phi$ is a phase function of order $\mu$ and $a \in \Sym{\Omega}{s}{m}$ and there is a $p \in \N_0$ such that $m + p \mu < -s$, then
\begin{equation}
\label{eq:D_phi}
 D_\phi(a)(x) = \int a(x, \theta) e^{i \phi(x, \theta)} \ud^s \theta \in C^p(\Omega)
\end{equation}
and the map $D_\phi: \Sym{\Omega}{s}{m} \to C^p(\Omega)$ is continuous.
\end{proposition}
\begin{proof}
For $m < -s$ we have
\[
 \sup_{x \in K} \betrag{D_\phi(a)(x)} \leq \norm{a}_{00K} \int (1+\betrag{\theta})^m \ud^s \theta \leq C \norm{a}_{00K},
\]
so that $\Sym{\Omega}{s}{m} \to C^0(\Omega)$ is continuous. Differentiation gives
\[
 \nabla_x D_\phi(a)(x) = \int \left\{ i \nabla_x \phi(x, \theta) a(x, \theta) + \nabla_x a(x, \theta) \right\} e^{i \phi(x, \theta)}  \ud^s \theta.
\]
The expression in curly brackets is a symbol of order $m+\mu$. With the same argument as before one can thus differentiate $p$ times.
\end{proof}


We formulate the main theorem of this section analogously to \cite[Thm.~1.11]{GrigisSjoestrand}. The proof is a straightforward generalization of the proof given there.
\begin{theorem}
\label{thm:Osc}
Let $\phi(x, \theta)$ be a phase function of order $\mu$ on $\Omega \times \R^s$. Then there is a unique way of defining $D_\phi(a) \in \D'(\Omega)$ for $a \in \Sym{\Omega}{s}{\infty}$ such that $D_\phi(a)$ coincides with \eqref{eq:D_phi} when $a \in \Sym{\Omega}{s}{m}$ for some $m < -s$ and such that, for all $m$, the map $D_\phi: \Sym{\Omega}{s}{m} \to \D'(\Omega)$ is continuous. Moreover, if $p \in \N_0$ and $m-p\mu < -s$, then the map $D_\phi : \Sym{\Omega}{s}{m} \to \D'^{(p)}(\Omega)$ is continuous.
\end{theorem}

To prove this, we need the following Lemma:

\begin{lemma}
\label{lemma:V}
Let $\phi$ be a phase function of order $\mu$ on $\Omega \times \R^s$. Then there exist $a_i \in \Sym{\Omega}{s}{-\mu+1}$, $b_j \in \Sym{\Omega}{s}{-\mu}$ and $c \in \Sym{\Omega}{s}{-\mu}$ such that for the differential operator 
\[
 Vf = a \cdot \nabla_\theta f + b \cdot \nabla_x f + cf
\]
with adjoint
\[
 V^tf = - \nabla_\theta \cdot (a f) - \nabla_x \cdot (b f) + cf
\]
we have
\begin{equation}
 \label{eq:Vt}
V^t e^{i \phi(x, \theta)} = e^{i \phi(x, \theta)}.
\end{equation}
Furthermore, $V$ is a continuous map from $\Sym{\Omega}{s}{m}$ to $\Sym{\Omega}{s}{m-\mu}$.
\end{lemma}
\begin{proof}
We choose $\eta$ as in Definition~\ref{def:Phase} and $\chi$ as in Remark~\ref{rem:Phase}. We  set
\begin{align*}
 a_i & = i \eta^{-1} (1-\chi) \betrag{\theta}^2 \del_{\theta_i} \phi, \\
 b_j & = i \eta^{-1} (1-\chi) \del_{x_j} \phi, \\
 c & = \nabla_\theta \cdot a + \nabla_x \cdot b + \chi.
\end{align*}
Then we have
\begin{align*}
 V^t e^{i \phi} & = \left( - \nabla_\theta \cdot a - i a \cdot \nabla_\theta \phi - \nabla_x \cdot b - i b \cdot \nabla_x \phi + \nabla_\theta \cdot a + \nabla_x \cdot b_j + \chi \right) e^{i\phi} \\
 & = \left( \eta^{-1} (1-\chi) \left( \betrag{\theta}^2 \betrag{\nabla_\theta \phi}^2 + \betrag{\nabla_x \phi}^2 \right) + \chi \right) e^{i\phi} \\
 & = e^{i\phi}.
\end{align*}
It is easy to see that $a_i$, $b_j$ and $c$ are symbols in the required way. The last statement follows from Proposition~\ref{prop:cont}.
\end{proof}

\begin{proof}[Theorem~\ref{thm:Osc}]
The uniqueness is a consequence of Proposition~\ref{prop:dense}. For $a \in \Sym{\Omega}{s}{-\infty}$ and $f \in C^\infty_0(\Omega)$, we have, with $V$ as in Lemma~\ref{lemma:V},
\begin{align*}
 \skal{D_\phi(a)}{f} & = \int e^{i \phi(x, \theta)} a(x, \theta) f(x) \ud^nx \ud^s\theta \\
& = \int (V^t)^p e^{i \phi(x, \theta)} a(x, \theta) f(x) \ud^nx \ud^s\theta \\
& = \int e^{i \phi(x, \theta)} V^p[a(x, \theta) f(x)] \ud^nx \ud^s\theta,
\end{align*}
for any $p \in \N_0$ and thus
\[
 \betrag{D_\phi(a)(f)} \leq \int \betrag{V^p[a(x, \theta) f(x)]} \ud^nx \ud^s\theta.
\]
Now the multiplication $\Sym{\Omega}{s}{m} \times C^\infty_0(\Omega) \to \Sym{\Omega}{s}{m}$ is continuous. Thus, if $a \in \Sym{\Omega}{s}{m}$, then $V^p[af]$ is a symbol of order $m - p \mu$ and in particular we have, for each $K \Subset \Omega$,
\begin{equation}
\label{eq:pBound}
 \sup_{K \times \R^s} \betrag{V^p[a(x, \theta) f(x)]} (1+\betrag{\theta})^{p\mu-m} \leq f_{p,K}(a) \sum_{\betrag{\alpha} \leq p} \sup_{x \in K} \betrag{D^\alpha f},
\end{equation}
where $f_{p,K}$ is a seminorm on $\Sym{\Omega}{s}{m}$. For $a \in \Sym{\Omega}{s}{m}$, we may thus choose $p$ such that $m - p \mu < -s$ and define
\[
 \skal{D^p_\phi(a)}{f} = \int e^{i \phi(x, \theta)} V^p[a(x, \theta) f(x)] \ud^nx \ud^s\theta.
\]
As the sum on the \rhs of \eqref{eq:pBound} is a seminorm on $\D'^{(p)}(\Omega)$, and due to the continuity properties discussed above, the map $D^p_\phi : \Sym{\Omega}{s}{m} \to \D'^{(p)}(\Omega)$ is continuous. For $q> p$, we have $\skal{D^q_\phi(a)}{f} = \skal{D^p_\phi(a)}{f}$ by \eqref{eq:Vt}. Thus, we can unambiguously define $D_\phi(a) = D^p_\phi(a)$.
\end{proof}

\section{The wave front set}
\label{sec:WF}
We may now further characterize the distributions that result from a generalized oscillatory integral.

\begin{definition}
\label{def:SP}
Let $\phi$ be a phase function of order $\mu$ on $\Omega \times \R^s$. We define
\begin{align*}
 \M(\phi) & = \{ (x,\theta) \in \Omega \times \dot \R^s | \nexists C,D> 0 \text{ s.t. } \betrag{\nabla_\theta \phi(x, \lambda \theta)} \geq C \lambda^{\mu-1} \ \forall \lambda > D \}, \\
 \SP(\phi) & = \{ (x, k) \in \Omega \times \dot \R^n | \exists (x,\theta) \in \M(\phi) \text{ s.t. } \nexists \alpha,D>0 \text{ s.t. } \\
 & \quad \sphericalangle(\nabla_x \phi(x,\lambda \theta),k) \geq \alpha \ \forall \lambda > D \}.
\end{align*}
We call $\SP(\phi)$ the \emph{asymptotic manifold of stationary phase}. By definition, it is conic.
\end{definition}

\begin{lemma}
\label{lemma:SP}
$\M(\phi)$ is a closed conic subset of $\Omega \times \dot \R^s$. $\SP(\phi)$ is a closed subset of $\Omega \times \dot \R^n$.
\end{lemma}
\begin{proof}
From the definition of $\M(\phi)$ it follows that if $(x, \theta) \in \M(\phi)$, then $(x, \lambda \theta) \in \M(\phi)$ for all $\lambda > 0$, so $\M(\phi)$ is conic. We now show that $(\Omega \times \dot \R^s) \setminus \M(\phi)$ is open in $\Omega \times \dot \R^s$. Let $(x_0, \theta_0)$ be such that there are positive $C, D$ such that
\[ 
  \betrag{\nabla_\theta \phi(x_0,\lambda \theta_0)} \geq C \lambda^{\mu-1} \ \forall \lambda > D.
\]
We set $F = \nabla_\theta \phi$. By Taylor's theorem we have
\begin{equation}
\label{eq:Taylor}
 F(x, \lambda \theta) = F(x_0, \lambda \theta_0) + (x-x_0) \cdot G(x, \lambda \theta) + \lambda (\theta - \theta_0) \cdot H(x, \lambda \theta),
\end{equation}
where $G$ and $H$ fulfill the bounds
\begin{align*}
 \betrag{G(x,\lambda \theta)} & \leq \sup_{y \in B_\epsilon(x_0), \tau \in B_{\epsilon}(\theta_0)} \betrag{\nabla_x F(y, \lambda \tau)}, \\
 \betrag{H(x,\lambda \theta)} & \leq \sup_{y \in B_\epsilon(x_0), \tau \in B_{\epsilon}(\theta_0)} \betrag{\nabla_\theta F(y, \lambda \tau)}.
\end{align*}
Here $(x, \theta)$ are chosen such that $\betrag{x-x_0} \leq \epsilon$, $\betrag{\theta - \theta_0} \leq \epsilon$. Furthermore, we restrict to $\epsilon < \betrag{\theta_0}$. Then we may use that $F$ is a symbol of order $\mu-1$ to conclude that there are positive constants $C_\epsilon, D_\epsilon$, which are bounded for $\epsilon \to 0$, for which
\begin{align*} 
 \betrag{G(x, \lambda \theta)} & \leq C_\epsilon \lambda^{\mu-1} \ \forall \lambda > D_\epsilon, \\ 
 \betrag{H(x, \lambda \theta)} & \leq C_\epsilon \lambda^{\mu-2} \ \forall \lambda > D_\epsilon. 
\end{align*}
holds. As the zeroth order term in the Taylor expansion \eqref{eq:Taylor} grows faster than $C \lambda^{\mu-1}$ for a fixed positive constant $C$, we can make $\epsilon$ so small that $\betrag{F(x,\lambda \theta)}$ grows faster than $C' \lambda^{\mu-1}$ for some positive $C'$. Thus, $\M(\phi)$ is closed in $\Omega \times \dot \R^s$.

In order to prove the closedness of $\SP(\phi)$, we first note that if $x_0 \notin \Pi_1 \M(\phi)$, then by the above there is a neighborhood of $x_0$ that does not intersect $\Pi_1 \M(\phi)$. Thus, it suffices to show that for $x_0 \in \M(\phi)$, $(x_0, k_0) \in (\Omega \times \dot \R^n) \setminus \SP(\phi)$ there is a neighborhood that does not intersect $\SP(\phi)$. By Condition~\ref{enum:Phase} of Definition~\ref{def:Phase}, there must be positive constants $C, D$ such that
\[
 \betrag{\nabla_x \phi(x_0, \theta)} \geq C \betrag{\theta}^\mu \ \forall \betrag{\theta}>D, (x_0, \theta) \in \M(\phi).
\]
By the same argument as above, such a bound holds true in a conic neighborhood $\Gamma$ of $\M(\phi) \cap \{x_0\} \times \dot \R^s$. By the definition of $\SP(\phi)$, there are posititve $\alpha, D$ such that
\begin{equation}
\label{eq:AngleBound}
 \sphericalangle(\nabla_x \phi(x_0, \lambda \theta_0),k_0) \geq \alpha \ \forall \lambda > D.
\end{equation}
We now want to show that one can choose a conic neighborhood $\Gamma_0$ of $(x_0, \theta_0)$, contained in $\Gamma$, such that an analogous bound holds, i.e., there are positive $\alpha, D$ so that
\begin{equation}
\label{eq:AngleBound2}
 \sphericalangle(\nabla_x \phi(x, \theta),k_0) \geq \alpha \ \forall \betrag{\theta} > D, (x, \theta) \in \Gamma_0.
\end{equation}
By the above construction, $\betrag{\nabla_x \phi(x, \lambda \theta)}$ grows as $\lambda^\mu$ in $\Gamma$. The deviations that occur by varying $x$ and $\theta$ also scale as $\lambda^\mu$, as $\phi$ is a symbol of order $\mu$. Recalling
\[
 \cos \sphericalangle(\nabla_x \phi(x, \theta),k_0) = \frac{\nabla_x \phi(x, \theta) \cdot k_0}{\betrag{\nabla_x \phi(x, \theta)} \betrag{k_0}}
\]
and again using Taylor's theorem, one shows that by making $\Gamma_0$ small enough, one still retains an inequality of the form \eqref{eq:AngleBound}. By suitably restricting $\Gamma_0$ in $\Omega$, we can ensure  that for $x \in \Pi_1 \Gamma_0 \cap \Pi_1 \M(\phi)$ we have $(x, \theta) \in \Gamma_0$ whenever $(x, \theta) \in \M(\phi)$. Then no new direction $\theta$ for which we would have to check the bound \eqref{eq:AngleBound2} can appear while varying $x$ in $\Pi_1 \Gamma_0$. Given \eqref{eq:AngleBound2}, it is clear that we can also take a conic neighborhood $\Gamma_1$ of $k_0$, by tilting it by angles less than $\alpha$. Choosing $U = \Pi_1 \Gamma_0 \times \Gamma_1$ gives a neighborhood of $(x_0, k_0)$ that does not intersect $\SP(\phi)$.
\end{proof}

\begin{proposition}
\label{prop:M}
If the support of the symbol $a$ does not intersect $\M(\phi)$, then $D_\phi(a)$ is smooth.
\end{proposition}

\begin{proof}
We choose a neighborhood $W$ of $\supp a$ whose closure does not intersect $\M(\phi)$. We choose a smooth function $\chi$ that is equal to one in a neighborhood of $\M(\phi)$ and vanishes on $\overline{W}$. We choose another smooth function $\psi$ on $\Omega \times \R^s$ with support in $W$ which is identical to one whenever $\betrag{\nabla_\theta \phi(x, \theta)} = 0$ for $(x, \theta) \in \supp a$. By definition of $\M(\phi)$, the set
\[
 N_x = \{ \theta \in \R^s | (x, \theta) \in \supp a, \betrag{\nabla_\theta \phi(x, \theta)} = 0 \}
\]
is bounded, so we can choose $\psi$ such that $\psi(x, \cdot)$ is compact for each $x$. Then we define
\begin{align*}
 \eta & = \betrag{\nabla_\theta \phi}^2, \\
 a_i & = i \eta^{-1} (1 - \chi) (1-\psi) \del_{\theta_i} \phi, \\
 c & = \nabla_\theta \cdot a + \psi, \\
 V & = a \cdot \nabla_\theta + c.
\end{align*}
By the definition of $\M(\phi)$ and $\chi$, we have $a_i \in \Sym{\Omega}{s}{-\mu+1}$ and $c \in \Sym{\Omega}{s}{-\mu}$. With these definitions, we have
\begin{equation*}
 V^t e^{i \phi} = \left( - \nabla_\theta \cdot a - i a \cdot \nabla_\theta \phi + \nabla_\theta \cdot a + \psi \right) e^{i\phi} = (1-\chi) e^{i\phi}.
\end{equation*}
Here we used that $\chi$ and $\psi$ have nonoverlapping supports. As $\chi a = 0$ and $V$ differentiates only \wrt $\theta$, we thus have $D_\phi(a) = D_\phi(V^p a)$ for arbitrary integer $p$. As $V$ maps symbols of order $m$ to symbols of order $m - \mu$, $D_\phi(a)$ is smooth by Proposition~\ref{prop:Diff}.
\end{proof}

\begin{corollary}
\label{cor:M}
The singular support of $D_\phi(a)$ is contained in $\Pi_1 \M(\phi)$.
\end{corollary}


\begin{theorem}
\label{thm:SP}
The wave front set of $D_\phi(a)$ is contained in $\SP(\phi)$.
\end{theorem}

\begin{lemma}
\label{lemma:SP1}
 Let $(x_0, \theta_0) \in \M(\phi)$, $(x_0, k_0) \in (\Omega \times \dot \R^n) \setminus \SP(\phi)$. Then there is a conic neighborhood $\Gamma$ of $(x_0, \theta_0)$, a conic neighborhood $V$ of $k_0$ and positive constants $C, D, \alpha$ such that
\begin{align}
\label{eq:C1}
 \betrag{\nabla_x \phi(x, \theta)} & \geq C \betrag{\theta}^\mu \ \forall \betrag{\theta} > D, (x, \theta) \in \Gamma, \\
\label{eq:C2}
 \sphericalangle(\nabla_x \phi(x, \theta), k) & \geq \alpha \ \forall \betrag{\theta} > D, (x, \theta) \in \Gamma, k \in V.
\end{align}
Furthermore, there is a positive $C'$ such that
\[
 \betrag{\nabla_x \phi(x, \theta) - k} \geq C' (\betrag{\theta}^\mu + \betrag{k}) \ \forall \betrag{\theta} > D, (x, \theta) \in \Gamma, k \in V.
\]
\end{lemma}

\begin{proof}
Condition \eqref{eq:C1} is fulfilled for $(x_0, \theta_0)$ by Condition~\ref{enum:Phase} of Definition~\ref{def:Phase}. That it is also fulfilled in a neighborhood of $(x_0, \theta_0)$ can be shown analogously to the proof of the closedness of $\M(\phi)$ in Lemma~\ref{lemma:SP}. Condition \eqref{eq:C2} is fulfilled for $(x_0, \theta_0, k_0)$. That it is also fulfilled in a neighborhood of $(x_0, \theta_0, k_0)$ can again be shown as in Lemma~\ref{lemma:SP}. In order to prove the last statement, we note that by \eqref{eq:C2} we have
\[
 \betrag{\nabla_x \phi(x, \theta) - k} \geq \betrag{q(\theta) - k} \ \forall \betrag{\theta} > D,
\]
where $q(\theta) \in \R^n$ has length $\betrag{\nabla_x \phi(x, \theta)}$ and lies on the cone with angle $\alpha$ around $k$ (see the figure, where $\nabla_x \phi(x, \theta)$ is denoted by $p$).
\begin{center}
 \includegraphics[scale=0.4]{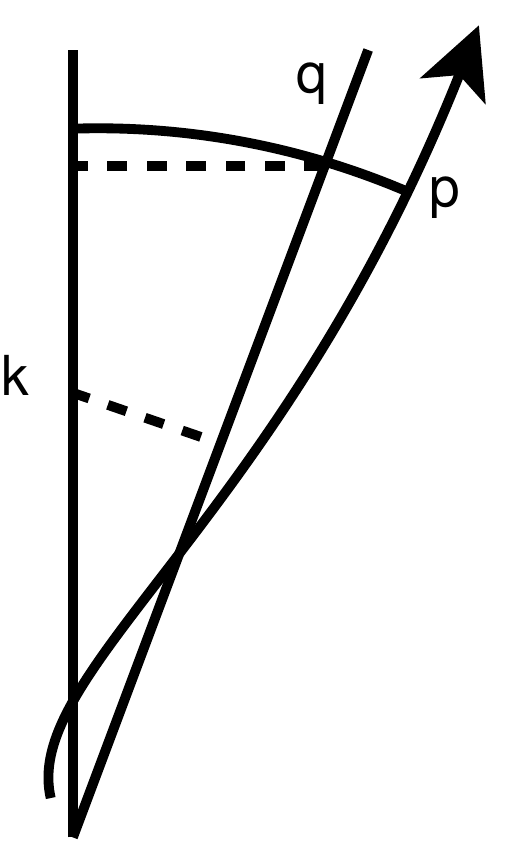}
\end{center}
For the distance of $k$ and $q(\theta)$ we have the bound (see the dashed lines in the figure)
\[
 \betrag{q(\theta)-k} \geq \tfrac{1}{2} \sin \alpha \left( \betrag{q(\theta)} + \betrag{k} \right).
\]
Using \eqref{eq:C1}, we then obtain the above statement.
\end{proof}

\begin{proof}[Theorem~\ref{thm:SP}]
By Corollary~\ref{cor:M}, it suffices to consider points $x_0 \in \Pi_1 \M(\phi)$.
Let $(x_0, k_0) \in (\Omega \times \dot \R^n) \setminus \SP(\phi)$. Due to Proposition~\ref{prop:M}, we may assume that $a$ is supported in an arbitrarily small closed conic neighborhood $\Gamma$ of $\M(\phi)$. We thus need to show that there is a $\psi \in C^\infty_0(\Omega)$, identically one near $x_0$ and a conic neighborhood $V$ of $k_0$ such that for each $N$ there is a seminorm $C_{m, N}$ on $\Sym{\Omega}{s}{m}$ such that
\begin{equation}
\label{eq:Goal}
 \betrag{\widehat{\psi D_\phi(a)}(k)} \leq C_{m,N}(a) (1+\betrag{k})^{-N} \ \forall N, k \in V, 
\end{equation}
for all $a$ supported in $\Gamma$. As in the proof of Theorem~\ref{thm:Osc}, it suffices to construct such a bound for $a \in \Sym{\Omega}{s}{-\infty}$ and then make use of Proposition~\ref{prop:dense}.

Let $\Gamma, V, D$ be as in Lemma~\ref{lemma:SP1}.
Choose $\zeta \in C^\infty_0(\Omega)$ that is identically one near $x_0$ and whose support is contained in $\Pi_1 \Gamma$. Choose a $\chi \in C^\infty_0(\R^s)$ that is identical to one on $\betrag{\theta} \leq D$. Now we set
\begin{align*}
 \eta & = \betrag{\nabla_x \phi - k}^2, \\
 b_j & = i \eta^{-1} \zeta (1-\chi) (\del_{x_j} \phi - k_{j}), \\
 c & = \nabla_x \cdot b, \\
 V_k & = b \cdot \nabla_x + c.
\end{align*}
Then for $k \in V$ we have
\begin{align*}
 V^t_k e^{i(\phi(x, \theta)-ik\cdot x)} & = \left( - \nabla_x \cdot b - i b \cdot (\nabla_x \phi - k) + \nabla_x \cdot b \right)  e^{i(\phi(x, \theta)-ik\cdot x)} \\
& = \zeta (1-\chi) e^{i(\phi(x, \theta)-ik\cdot x)}.
\end{align*}
Now we choose $\psi \in C_0^\infty(\Omega)$, identical to one near $x_0$ and with support in $\zeta^{-1}(1)$. We also choose $\xi \in C^\infty_0(\R^s)$ which is identically one on $\supp \chi$. We consider $a = (1-\xi) a + \xi a$. Then by Proposition~\ref{prop:Diff}, the second term yields a smooth function, 
so that the above bound is fulfilled. It remains to consider
\begin{align*}
 \betrag{\widehat{\psi D_\phi((1-\xi)a)}(k)} & = \betrag{\int e^{i\phi(x, \theta) - ik\cdot x} \psi(x) (1-\xi(\theta)) a(x, \theta) \ud^n x \ud^s \theta} \\
 & = \betrag{\int e^{i\phi(x, \theta) - ik\cdot x} V_k^p[\psi(x) (1-\xi(\theta)) a(x, \theta)] \ud^n x \ud^s \theta} \\
 & \leq  \int \betrag{ V_k^p[\psi(x) (1-\xi(\theta)) a(x, \theta)]} \ud^n x \ud^s \theta.
\end{align*}
Here we used that $\zeta (1-\chi)$ is identically one on the support of $\psi (1-\xi)$. By Lemma~\ref{lemma:SP1}, one now has
\[
 \betrag{ V_k^p[\psi(x) (1-\xi(\theta)) a(x, \theta)]} \leq C_{m,p}(a) (1+\betrag{\theta})^m (\betrag{\theta}^\mu + \betrag{k})^{-p},
\]
where $C_{m,p}$ is a seminorm on $\Sym{\Omega}{s}{m}$. By using
\begin{align*}
 (\betrag{\theta}^\mu + \betrag{k})^{-1} & \leq (D^\mu + \betrag{k})^{-1}, \\
 (\betrag{\theta}^\mu + \betrag{k})^{-1} & \leq \betrag{\theta}^{-\mu} \leq (\tfrac{D}{2} + \tfrac{D}{2} \betrag{\theta})^{-\mu},
\end{align*}
one can make the integral convergent and assure \eqref{eq:Goal} by choosing $p$ large enough.
\end{proof}

\section{Applications}

\begin{example}
\label{ex:Delta+}
We consider the two-point function $\Delta_+$ of a free massive scalar relativistic field, where one has $\Omega = \R^4, s = 3$ and
\begin{equation}
\label{eq:phiKG}
 \phi(x, \theta) = - x_0 \omega(\theta) + \V{x} \cdot \theta,
\end{equation}
with 
\begin{equation}
\label{eq:omega}
 \omega(\theta) = \sqrt{\betrag{\theta}^2 + m^2}.
\end{equation}
Here, we use the notation $x = (x_0, \V{x})$, with $\V{x} \in \R^3$. 
Note that $\phi$ is not homogeneous. In \cite{RS2} this problem is circumvented by using $-x_0 \betrag{\theta} + \V{x} \cdot \theta$ as phase function and multiplying the symbol with the function $e^{ix_0 (\betrag{\theta} - \omega(\theta))}$. It is then no longer a symbol (as it is not smooth in $\theta = 0$), so one has to allow for so-called asymptotic symbols. Furthermore, one has to show that the multiplication with such a term does not spoil the fall-off properties, in particular that differentiation \wrt $\theta$ lowers the order.

In the present approach, this is not necessary. 
$\phi$ is a phase function in the sense of Definition~\ref{def:Phase} and therefore, by Theorem~\ref{thm:Osc}, it defines an oscillatory integral for every symbol $a$.
In order to find the wave front set, we compute
\[
 \nabla_\theta \phi(x, \theta) = - \frac{\theta}{\omega(\theta)} x_0 + \V{x}.
\]
It is easy to see that its modulus is bounded from below by a positive constant unless $x=0$ or $\betrag{x_0} = \betrag{x}$ and $\V{x} = \sign(x_0) \theta$. Thus, we have
\[
 \M(\phi) = \{(x, \theta) | x = 0 \text{ or } \betrag{x_0} = \betrag{x} \neq 0, \theta = \lambda \V{x} / x_0, \lambda > 0 \}.
\]
Furthermore,
\[
 \nabla_x \phi(x, \theta) = (-\omega(\theta), \theta).
\]
For large $\theta$, this behaves as
\[
 \nabla_x \phi(x, \theta) \sim (-\betrag{\theta}, \theta) + R(\theta),
\]
where the remainder term $R$ scales as $\betrag{\theta}^{-1}$. Thus, only in the directions $(-\betrag{\theta}, \theta)$ the bound on the angle of $k$ and $\nabla_x \phi$ can not be fulfilled. Hence, we obtain the well-known result\footnote{In the physical literature, the Fourier transform is computed by integration with $e^{i(k_0 x_0 - \V{k} \cdot \V{x})}$. In that convention, the sign of the zeroth component in the cotangent bundle has to be reversed.}
\[
 \WF(\Delta_+) \subset \SP(\phi) = \{ (0, (-\betrag{\V{k}}, \V{k})) \} \cup \{ ((\pm \betrag{\V{x}}, \V{x}), (-\lambda \betrag{\V{x}}, \pm \lambda \V{x})) | \lambda > 0 \}.
\]

A variant of this example is obtained by considering phase functions of the form
\begin{equation}
\label{eq:phiDistorted}
 \phi(x, \theta) = - x_0 \sqrt{\betrag{\theta}^2 + f(\theta)} + \V{x} \cdot \theta,
\end{equation}
with $f$ a positive function that is a symbol of order $\nu < 2$. Such expressions occur for example in quantum field theory on the Moyal plane with hyperbolic signature and signify a distortion of the dispersion relations, cf. \cite{Quasiplanar,NCDispRel}. The above trick to put $e^{-ix_0(\sqrt{\betrag{\theta}^2 + f(\theta)} - \betrag{\theta})}$ into the symbol still works, but then the symbol will be of type $\rho' = \min(\rho, 2-\nu)$, where $\rho$ is original type of the symbol. It is straightforward to check that \eqref{eq:phiDistorted} still defines a phase function of order 1 in the sense defined here, and that its stationary phase is as above. 

If the function $f$ in \eqref{eq:phiDistorted} is a symbol of order $\nu > 2$,\footnote{The phase function $\phi$ of \eqref{eq:phiKG} corresponds to solutions of the hyperbolic wave equation $(\del_t^2 - \Delta_x + m^2) \psi = 0$. The modification suggested here means that the underlying PDE is no longer hyperbolic.}  then the shift of $e^{-ix_0(\sqrt{\betrag{\theta}^2 + f(\theta)} - \betrag{\theta})}$ to the symbol is not possible, as this would no longer give a symbol of type $\rho > 0$. Thus, a treatment of \eqref{eq:phiDistorted} in the context of phase functions that are homogeneous of order 1 is not possible. However, one can still interpret \eqref{eq:phiDistorted} as a phase function of order $\nu/2$ and easily computes
\[
 \SP(\phi) = \{ (x, k) | x_0 = 0, \V{k} = 0, k_0 \neq 0 \}.
\]
\end{example}


\begin{example}
\label{ex:NCQFT}
In quantum field theory on the Moyal plane of even dimension $d$ with Euclidean signature\footnote{For an overview on the subject, we refer to \cite{ReviewWulkenhaar}.}, one frequently finds phase funtions of the form
\[
 \R^d \times \R^{2d} \ni (x, \theta_1, \theta_2) \mapsto \phi(x,\theta) = x \cdot (\theta_1 + \theta_2) + \theta_1^t \sigma \theta_2.
\]
Here $\sigma$ is some real antisymmetric $d \times d$ matrix of maximal rank $d$. The above is clearly a symbol of order 2, and we have
\begin{align*}
 \nabla_{\theta_1} \phi(x, \theta) & = x + \sigma \theta_2, \\
 \nabla_{\theta_2} \phi(x, \theta) & = x - \sigma \theta_1.
\end{align*}
As $\sigma$ has rank $d$, Condition~\ref{enum:Phase} of Definition~\ref{def:Phase} is fulfilled. From the above it follows that $\M(\phi) = \emptyset$ and thus also $\SP(\phi) = \emptyset$, so that the resulting distributions are smooth.

We note that up to now such integrals could only be treated in the so-called adiabatic limit \cite{Doscher}. But then one loses the information about the singular behaviour in position space, contrary to the present case, where the wave front set is known completely. 
\end{example}

\begin{example}
\label{ex:NCQFTb}
In quantum field theory on the Moyal plane with hyperbolic signature, one frequently finds phase functions of the form
\[
 \R^d \times \R^{2(d-1)} \ni (x, \theta_1, \theta_2) \mapsto \phi(x,\theta) = x \cdot (\tilde \theta_1 + \tilde \theta_2) + \tilde \theta_1^t \sigma \tilde \theta_2,
\]
where $\tilde \theta = (\omega(\theta), \theta)$ with $\omega$ as in \eqref{eq:omega} and $\sigma$ as in Example~\ref{ex:NCQFT}. The above is a symbol of order 2, but it is not a phase function as defined here, as can most easily be seen in the case $d=2$. Then with $\sigma = \epsilon$ one obtains
\begin{align*}
 \del_{\theta_1} \phi(x, \theta) & = x_0 \frac{\theta_1}{\omega(\theta_1)} + x_1 + \theta_2 \frac{\theta_1}{\omega(\theta_1)} - \omega(\theta_2), \\
 \del_{\theta_2} \phi(x, \theta) & = x_0 \frac{\theta_2}{\omega(\theta_2)} + x_1 - \theta_1 \frac{\theta_2}{\omega(\theta_2)} + \omega(\theta_1).
\end{align*}
If the signs of $\theta_1$ and $\theta_2$ coincide, then the above derivatives tend to a constant as a function of $\theta$, so that Condition~\ref{enum:Phase} of Definition~\ref{def:Phase} is not fulfilled. The rigourous treatment of such integrals is an open problem, which we plan to address in future work\footnote{Joint work with Dorothea Bahns.}.
\end{example}


{\bf \noindent Acknowledgements}

\noindent It is a pleasure to thank Dorothea Bahns for helpful discussions and her detailed comments on the manuscript. I would also like to thank Ingo Witt for valuable comments. This work was supported by the German Research Foundation (Deutsche Forschungsgemeinschaft (DFG)) through the Institutional Strategy of the University of G\"ottingen.



\end{document}